\documentclass[11pt]{amsart}

\usepackage{amsmath,amssymb,xspace}

\newcommand*{\C}{\ensuremath{\mathbb C}\xspace}

\newcommand*{\Ii}{\ensuremath{\mathcal I}\xspace}
\newcommand*{\m}{\ensuremath{\mathfrak m}\xspace}

\newcommand*{\Oo}{\ensuremath{\mathcal O}\xspace}
\newcommand*{\PP}{\ensuremath{\mathbb P}\xspace}
\newcommand*{\Q}{\ensuremath{\mathbb Q}\xspace}
\newcommand*{\R}{\ensuremath{\mathbb R}\xspace}

\newcommand*{\Z}{\ensuremath{\mathbb Z}\xspace}

\DeclareMathOperator{\rank}{rank}
\DeclareMathOperator{\Pic}{Pic}
\DeclareMathOperator{\Num}{Num}
\DeclareMathOperator{\NE}{NE}
\DeclareMathOperator{\NEb}{\overline{NE}}

\newtheorem{proposition}{Proposition}[section]

\newtheorem{lemma}[proposition]{Lemma}

\newtheorem{corollary}[proposition]{Corollary}

\theoremstyle{remark}
\newtheorem{remark}[proposition]{Remark}

\theoremstyle{definition}
\newtheorem{definition}[proposition]{Definition}

\title{On non-projective small resolutions}

\author{Serge Lvovski}

\address{National Research University Higher School of Economics,
  Russian Federation \hfil\break
Federal Scientific Centre Science Research Institute 
of System Analysis at  Russian Academy of Science 
(FNP FSC SRISA  RAS)}
\email{lvovski@gmail.com}



\thanks{Partially supported by HSE University Basic Research Program and
by Simons--IUM fellowship. The work was done according to FSI SRISA RAS research project No. 0580-2021-0007 (Reg. No 121031300051-3).}

\begin{document}

\begin{abstract}
We construct a large class of projective threefolds with one node (aka
non-degenerate quadratic singularity) such that their small
resolutions are not projective.
\end{abstract}

\maketitle

\section{Introduction}

One of the main results of this note is the following

\begin{proposition}\label{omega(1)}
Suppose that $X\subset\PP^n$ is a smooth $4$-dimensional projective
variety over~\C such that $H^0(X,\omega_X(1))\ne 0$, where $\omega_X$
is the canonical sheaf. Then a general singular hyperplane section
$Y_0\subset X$ has precisely one singular point~$a\in Y_0$, this point
is a node, and small resolutions of the point~$a$ cannot be projective
varieties.
\end{proposition}

\begin{corollary}\label{omega(k)}
Suppose that $X\subset\PP^n$ is an arbitrary smooth $4$-dimensional projective
variety over~\C. Then there exists a natural number $m_0$ such that,
for any $m\ge m_0$, a general singular intersection of $X$ with a
hypersurface of degree~$m$ has precisely one singular
point, this point is a node, and small resolutions of this node
cannot be projective varieties. 
\end{corollary}

\begin{proof}
Take an $m_0$ such that $H^0(\omega_X(m))\ne 0$ for all $m\ge m_0$ and
apply Proposition~\ref{omega(1)} to the $m$-th Veronese image of~$X$.
\end{proof}

\begin{remark}
In view of Proposition~\ref{not-A} below, this corollary follows also
from Proposition~\ref{omega(1)} and~\cite[Expos\'e~XVIII, Corollaire
  6.4]{SGA7.2}.
\end{remark}

\begin{remark}
For the case in which the fourfold $X\subset\PP^n$ is a complete
intersection, the result of Corollary~\ref{omega(k)} is contained in
Theorem~B from the paper~\cite{PRS} by Polizzi, Rapagnetta, and
Sabatino. Indeed, this theorem from~\cite{PRS} implies that if $X$ is
a complete intersection and $Y$ is the intersection of $X$ with a
hypersurface of degree~$\ge3$ and if the only singular point of $Y$ is
a node, then the projective variety~$Y$ is factorial. On the other
hand, the factoriality, and even \Q-factoriality, of a nodal threefold
implies that it has no projective small resolution
(see~\cite{Cheltsov} or Remark~\ref{rem:cheltsov} below). Thus, for
complete intersections this result from~\cite{PRS} provides an
explicit bound $m_0=3$ in Corollary~\ref{omega(k)}. Besides,
\cite[Theorem~B]{PRS} provides for the case of finitely many ordinary
singularities of higher multiplicity on~$Y$.
\end{remark}

We derive Proposition~\ref{omega(1)} from a stronger result. To
state it, recall some terminology from SGA7.

\begin{definition}\label{def:A}
Suppose that $X\subset\PP^n$ is a smooth projective variety over~\C.
We will say that \emph{Condition} (A) \emph{is satisfied for the
  variety~$X$} if either $X^*\subset(\PP^n)^*$ is not a hypersurface
or $X^*$ \emph{is} a hypersurface and the restriction homomorphism
$H^{d-1}(X,\Q)\to H^{d-1}(Y,\Q)$, where $d=\dim X$ and $Y\subset X$ is
some (equivalently, any) smooth hyperplane section, is not a
surjection.
\end{definition}
(The restriction $H^{d-1}(X,\Q)\to H^{d-1}(Y,\Q)$ is always an
injection by virtue of Lefschetz hyperplane section theorem.)

It follows from \cite[Expos\'e~XVIII, Theorem 6.3]{SGA7.2} and hard
Lefschetz theorem that in characteristic zero Definition~\ref{def:A}
is equivalent to the definition from Section~5.3.5
of~\cite[Expos\'e~XVIII]{SGA7.2}, where Condition~(A) was originally
stated.

Recall also (\emph{loc. cit.}) that Condition~(A) can be violated only
if $d=\dim X$ is even and that, for an even-dimensional
variety~~$X\subset\PP^n$ \emph{such that $X^*$ is a hypersurface},
Condition~(A) is equivalent to the non-triviality of the monodromy
group acting on $H^{d-1}(Y,\Q)$ (if $X^*$ is not a hypersurface, this
monodromy group is automatically trivial).

Taking all this into account, here is the promised stronger result.

\begin{proposition}\label{not-A}
Suppose that $X\subset\PP^n$ is a smooth $4$-dimensional projective
variety over~\C. 

If Condition~\textup{(A)} holds for~$X$ and $X^*\subset(\PP^n)^*$ is a
hypersurface, then a general singular hyperplane section $Y_0\subset
X$ has precisely one singular point~$a\in Y_0$, this point is a node,
and small resolutions of $Y_0$ cannot be projective varieties.

If Condition~\textup{(A)} does not hold for~$X$, in which case $X$ is
automatically a hypersurface, then a general singular
hyperplane section $Y_0\subset X$ has precisely one singular
point~$a\in Y_0$, this point is a node, and there exists a 
small resolution for~$Y_0$ in the category of projective varieties.
\end{proposition}

It is believed (see~\cite[Expos\'e~XVIII, 5.3.5]{SGA7.2}) that 
Condition~(A) holds for ``most'' varieties of even dimension. In small
dimensions one can hope for a complete classification of the exceptions.
In dimension~$2$, varieties for which Condition~(A) is violated were 
classified by Zak~\cite{Zak}. As far as I know, no classification in
dimension~$4$ has been found yet.

The paper is organized as follows. In Section~\ref{sec:contration} we
state and prove several well-known properties of small resolutions for
which I did not manage to find references; thus, the results in this
section do not claim to much novelty. In Section~\ref{sec:top2geom} we
establish the link between Condition~(A) and properties of small
resolutions of singular hyperplane sections of a smooth fourfold.
Finally, we prove Propositions~\ref{omega(1)} and~\ref{not-A} in
Section~\ref{conclusion}.

\subsection*{Acknowledgements}

I am grateful to Ivan Cheltsov and Constantin Shramov for useful
discussions, and to Alexander Braverman for communicating to me
Lemma~\ref{IClemma} with an idea of its proof. It is a pleasure to
thank the anonymous referees for numerous helpful suggestions, which
contributed much to improvement of the text.

\section*{Notation and conventions}

Base field will be the field \C of complex numbers.

By \emph{node} on an algebraic threefold~$W$ we mean a point $a\in W$
such that $\widehat{\Oo_{a,W}}\cong \C[[x,y,z,t]]/(x^2+y^2+z^2+t^2)$
(a non-degenerate quadratic singularity). If $\sigma\colon\tilde W\to
W$ is the blowup of~$a$ and $Q=\sigma^{-1}(a)$, then $\tilde W$ is
non-singular along~$Q$, $Q\cong \PP^1\times\PP^1$, and the normal
sheaf $\Oo_Q(Q)$ is isomorphic to
$\Oo_Q(-1,-1)=\mathrm{pr}_1^*\Oo_{\PP^1}(-1)\otimes
\mathrm{pr}_2^*\Oo_{\PP^1}(-1)$.

If $X$ is a topological space and $Y\subset X$, then $X/Y$ stands for
the topological space obtained from $X$ by contraction of $Y$ to a point.

If $X$ is a smooth projective variety, then $N_1(X)=\Num(X)\otimes\R$,
where $\Num(X)$ is the group of $1$-cycles modulo numerical
equivalence; $\NE(X)\subset N_1(X)$ (resp.\ $\NEb(X)\subset N_1(X)$)
is the cone (resp.\ the closed cone) of effective $1$-cycles on $X$
modulo numerical equivalence. If $C\subset X$ is an effective curve,
then its class in $N_1(X)$ is denoted by~$[C]$.

The notation for the Betti numbers of a topological space~$X$ is~$b_i(X)$.

If $W$ is a $3$-dimensional algebraic variety with isolated
singularities then by a \emph{small resolution} we mean a proper
morphism of complex varieties $\pi\colon W_1\to W$ such that $W_1$ is
smooth, $\pi$ is an isomorphism over the smooth locus of~$W$, and the
fibers over singular points of~$W$ have dimension at most~$1$.

\section{Small resolutions}\label{sec:contration}

The results of this section seem to be well known, but I did not
manage to find appropriate references.

\begin{proposition}\label{b_2=b_4}
Suppose that $W$ is a projective threefold with a unique singular
point~$a$ that is a node, and let $\sigma\colon \tilde W\to W$ be the
blowup of the point~$a$. Put $Q=\sigma^{-1}(a)\subset\tilde W$ and let
$i_*\colon H_2(Q,\Q)\to H_2(\tilde W,\Q)$ be the homomorphism induced
by the embedding $Q\hookrightarrow \tilde W$.

In this setting, either $\rank i_*=1$ and $b_2(W)=b_4(W)$, or
$\rank i_*=2$ and $b_2(W)=b_4(W)-1$.
\end{proposition}

\begin{proof}
Since $Q\cong \PP^1\times\PP^1$ is a projective subvariety of $\tilde
W$, $\rank i_*\ge1$. Since $W$ is homeomorphic to $\tilde W/Q$, we may
consider the following fragments of the exact sequence of the pair
$(\tilde W,Q)$:
\begin{gather}
  H_4(Q,\Q)\to H_4(\tilde W,\Q)\to
  H_4(W,\Q)\to H_3(Q,\Q)\label{b_2(tildeY).1}\\
H_2(Q,\Q)\xrightarrow{i_*} H_2(\tilde W,\Q)\to H_2(W,\Q)
\to H_1(Q,\Q),\label{b_2(tildeY).2}
\end{gather}
It follows from \eqref{b_2(tildeY).1} that $b_4(W)=b_4(\tilde W)-1$,
and it follows from \eqref{b_2(tildeY).2} that $b_2(W)=b_2(\tilde
W)-\rank i_*$. Since the smoothness of $\tilde W$ implies that
$b_2(\tilde W)=b_4(\tilde W)$, the proposition follows.\qed
\end{proof}

\begin{remark}\label{rem:cheltsov}
The condition $b_2(W)=b_4(W)$ for nodal projective threefolds~$W$ was
considered by I.\,Cheltsov in~\cite{Cheltsov}. In particular, Cheltsov
observes that this condition is equivalent, at least in some important
cases, to the \Q-factoriality of~$W$ and that the \Q-factoriality of a
nodal threefold implies that it has no projective small resolutions.
\end{remark}

Again, let $W$ be a projective threefold with a unique singular
point~$a$, and suppose that $a$ is a node.

\begin{proposition}\label{res-nores}
In the above setting, let $\sigma\colon \tilde W\to W$ be the blowup
of $W$ at~$a$, and put $Q=\sigma^{-1}(a)\subset \tilde W$. Let
$i_*\colon H_2(Q,\Q)\to H_2(\tilde W,\Q)$ be the homomorphism induced
by the embedding $Q\subset\tilde W$. Then the following implications
hold.

\textup{(i)} If $\rank i_*=1$, then the
point~$a$ does not admit a small resolution in the category of projective
varieties.

\textup{(ii)} If $\rank i_*=2$, then the
point~$a$ admits a small resolution in the category of projective
varieties.
\end{proposition}

We begin with some lemmas.

\begin{lemma}\label{H^1(O)=0}
Suppose that $W$ is a $3$-dimensional algebraic variety with a unique
singular point $a$ that admits a small resolution $\pi\colon W_1\to
W$, where $W_1$ is an algebraic variety; put $C=\pi^{-1}(a)$
\textup(set-theoretically\textup). Besides, suppose that the
singularity of $W$ at the point~$a$ is rational.

Then $H^1(C,\Oo_C)=0$.
\end{lemma}

\begin{proof}
Arguing by contradiction, assume that $H^1(C,\Oo_C)\ne0$.  Denote by
$\tilde C=\pi^*a$ the scheme-theoretical fiber over~$a$, so that
$\tilde C_{\mathrm{red}}\cong C$. Since there exists a surjection
$\Oo_{\tilde C}\to \Oo_{C}$ and since the functor $H^1(\cdot)$ is
right exact on one-dimensional schemes, one has $H^1(\tilde
C,\Oo_{\tilde C})\ne0$. If $\Ii\subset\Oo_{W_1}$ is the ideal sheaf of
the closed subscheme $\tilde C\subset W_1$, then, by the same right
exactness, $H^1(\Oo_{W_1}/\Ii^{n+1})$ surjects onto
$H^1(\Oo_{W_1}/\Ii^{n})$ for any natural $n$. Hence, $\varprojlim
H^1(\Oo_{W_1}/\Ii^{n})$ surjects onto $H^1(\Oo_{\tilde C})\ne 0$. So,
$\varprojlim H^1(\Oo_{W_1}/\Ii^{n})\ne0$. By Theorem on formal
functions (see for example \cite[Theorem 11.1]{H}) this implies that
the $\m_a$-adic completion of the stalk of $R^1\pi_*\Oo_{W_1}$ at~$a$
is not zero. Hence, the stalk itself is not zero either, which
contradicts the hypothesis that the singularity of $W_1$ at~$a$ is
rational.\qed
\end{proof}

\begin{lemma}\label{IClemma}
Suppose that $W$ is a projective threefold with a unique
singular point~$a$ and that $\pi\colon W_1\to W$ is its small resolution
in the category of complex spaces.

Then the singular cohomology $H^k(\pi^{-1}(a),\Q)$ is independent of
the small resolution~$\pi$, for all~$k$.
\end{lemma}

\begin{proof}
This proof is an adaptation of a less elementary (but quicker) proof
suggested by A.\,Braverman.

In the proof we will work with sheaves and complex varieties in the
classical topology. Suppose that $\pi\colon W_1\to W$ is a small
resolution of $W$ and put $C=\pi^{-1}(a)$. The
morphism $\pi$ induces an isomorphism between $W_1\setminus C$ and
$W\setminus \{a\}$; we denote this complex manifold
by~$U$. Let $i\colon U\hookrightarrow W$, $j\colon U\hookrightarrow
W_1$ be the natural embeddings. Let $\Q_U$ (resp. $\Q_{W_1}$) be the
constant sheaf with the stalk~\Q on $U$ (resp.~$W_1$). Since
$\pi\circ j=i$, there exists a first quadrant spectral sequence
\begin{equation}\label{specseq}
E_2^{pq}=R^p\pi_*(R^qj_*\Q_U)\Rightarrow R^{p+q}i_*\Q_U  
\end{equation}
(see \cite[Theorem III.8.3e]{GelMan}, where we put $\mathcal R_X=\Z$
(the constant sheaf)).
Since $C$ is a $1$-dimensional complex subspace in the $3$-dimensional
smooth complex manifold $W_1$, one has $R^0j_*\Q_U=\Q_{W_1}$ and
$R^qj_*\Q_U=0$ for $q=1,2$, so $E_2^{pq}=0$ for $0<q<3$.
Taking these vanishings
into account, one infers from~\eqref{specseq} that
$R^k\pi_*\Q_{W_1}\cong R^ki_*\Q_U$ for $0\le k\le 2$. Besides,
$R^k\pi_*\Q_{W_1}=0$ for $k>2$ since the
fibers of $\pi$ are at most (complex) $1$-dimensional. Thus, the sheaves
$R^k\pi_*\Q_{W_1}$ are independent of the choice of~$\pi$.

It remains to observe that for any natural~$k$ the stalk of
$R^k\pi_*\Q$ at~$a$ is isomorphic to $H^k(\pi^{-1}(a),\Q)$, so this
cohomology is independent of the choice of~$\pi$ as well.\qed
\end{proof}

\begin{remark}
Actually, it is the direct image $R\pi_*\Q_{W_1}$ that is independent
of the choice of $\pi$. To wit, the latter direct image is
isomorphic (in the derived category of constructible sheaves on~$W$)
to $\tau_{\le2}Ri_*\Q_{U}$, where $\tau_{\le2}$ is the truncation functor.
\end{remark}

\begin{lemma}\label{blowuplemma}
Let $W$ be a three-dimensional projective variety, smooth except for a
point $a\in W$ that is a node, and suppose that $f\colon Z\to W$ is a
projective birational morphism such that $Z$ is smooth and $f$ induces
an isomorphism between $Z\setminus f^{-1}(a)$ and
$W\setminus\{a\}$. Put $f^{-1}(a)=F$
\textup(set-theoretically\textup).

If $F\cong\PP(\Oo_{\PP^1}\oplus\Oo_{\PP^1}(-n))$ for some $n\ge0$, then
$F\cong\PP^1\times\PP^1$ and the mapping~$f$ is isomorphic to the
blowup $\sigma\colon \tilde W\to W$ of the point $a\in W$.
\end{lemma}

\begin{proof}
Let $\Ii_a\subset\Oo_W$ be the ideal sheaf of the point~$a$. Since $Z$
is smooth and the zero locus of the ideal $f^{-1}\Ii_a\cdot\Oo_Z$
coincides with the surface~$F\subset Z$, one has
$f^{-1}\Ii_a\cdot\Oo_Z=\Oo_Z(-mF)$ for some $m>0$, which is an
invertible ideal sheaf. Now it follows from the universality of blowup
\cite[Proposition II.7.14]{H} that there exists a morphism $g\colon
Z\to\tilde W$, where $\sigma\colon \tilde W\to W$ is the blowup of
$W$ at~$a$, such that $\sigma\circ g=f$.

Let $Q\subset \tilde W$ be the exceptional divisor of the
blowup~$\sigma$. Since $f$ induces an isomorphism between $Z\setminus
f^{-1}(a)$ and $W\setminus\{a\}$ and $\sigma$ induces an isomorphism
between $\tilde W\setminus Q$ and $W\setminus\{a\}$, one concludes
that $g$ induces an isomorphism between $Z\setminus F$ and $\tilde
W\setminus Q$. Since $g(Z\setminus F)=\tilde W\setminus Q$ and $g(Z)$
is closed in $W$, one has~$g(F)=Q$. We claim that $n=0$. 
Arguing by contradiction, suppose that $n>0$.

Identify $Q$ with the smooth quadric in $\PP^3$. Let $\ell,m\subset Q$
be generators of the quadric, from different families, and let
$g'\colon F\to Q$ be the restriction of~$g$ to~$Q$. Put $L=(g')^*\ell$
and $M=(g')^*m$ (the pullback divisors on $F$). Since the
self-intersections $(\ell,\ell)=(m, m)=0$ and the complete
linear systems $|\ell|$ and~$|m|$ on~$Q$ have no basepoints, we
conclude that $(L, L)=(M, M)=0$ and the complete linear
systems $|L|$ and~$|M|$ have no basepoints either.

Since $F\cong\PP(\Oo_{\PP^1}\oplus\Oo_{\PP^1}(-n))$, one
has $\Pic(F)=\Z s\oplus\Z f$, where $f$ is the class of a fiber and
the intersection indices are $(s,s)=-n$, $(s,f)=1$, $(f,f)=0$
(see~\cite[Proposition 2.9]{H}). Now suppose that $D\sim as+bf$ is an
effective divisor on~$F$ such that $(D,D)=0$ and $|D|$ has no
basepoints. We claim that such a $D$ is linearly equivalent to a
multiple of a fiber. Indeed, it follows from our assumptions that
\begin{equation}\label{eq:ruled}
  \begin{aligned}
    (D,D)&=-na^2+2ab=0,\\
    (D,s)&=-na+b\ge0,\\
     (D,f)&=a\ge 0.
  \end{aligned}
\end{equation}
If $a>0$, then it follows from \eqref{eq:ruled} that $b=na/2$ and
$-na/2\ge0$, whence $n=0$, which contradicts our assumption $n>0$; if
$a=0$, then $D\sim bf$, which proves our claim.

It follows from what we have just proved that both the divisors $L$
and $M$ are linearly equivalent to a multiple of the fiber, and so is
$L+M$. Since $\Oo_F(L+M)=(g')^*\Oo_Q(1)$, this contradicts the fact
that $g'(F)$ is two-dimensional.

Thus, $n=0$ and $F$ is isomorphic to the smooth two-dimensional
quadric. It follows from the fact that the two-dimensional quadric
does not contain exceptional curves that $g'\colon F\to Q$ has finite
fibers, so the morphism $g$ has finite fibers, so $g$ is finite. Since
$g$ is birational and $\tilde Y$ is smooth, $g$ is an isomorphism.\qed
\end{proof}

\begin{proof}[of assertion~\textup{(i)} of
    Proposition~\ref{res-nores}] 
Given that $\rank i_*=1$, we have to prove that a projective small
resolution of the point~$a\in W$ does not exist.

Arguing by contradiction, suppose that $\pi\colon W_1\to W$, where
$W_1$ is a projective variety, is a small resolution, and put
$C=\pi^{-1}(a)$ (set-theoretically). Lemma~\ref{H^1(O)=0} implies that
$H^1(C,\Oo_C)=0$. There exists, in the category of complex spaces, a
small resolution of the node~$a$ for which the fiber over~$a$ is
isomorphic to $\PP^1$ (see for example \cite[Lemma 4]{Atiyah}
or~\cite{Reid}). Now Lemma~\ref{IClemma} implies that $H^2(C,\Q)\cong
H^2(\PP^1,\Q)$ (singular cohomology), so $C$ is irreducible. Putting
these two facts together, one concludes that $C\cong \PP^1$.

Let $s\colon Z\to W_1$ be the blowup of the curve~$C$ in $W_1$. Since
$C\cong\PP^1$, one has $s^{-1}(C)\cong
F=\PP(\Oo_{\PP^1}\oplus\Oo_{\PP^1}(-n))$ for some~$n\ge0$. The
composition $\pi\circ s\colon Z\to W$ blows down the surface $F$, so,
by Lemma~\ref{blowuplemma} above, $F$ is isomorphic to the smooth
two-dimensional quadric and $\pi\circ s\colon Z\to W$ is isomorphic to
$\sigma\colon \tilde W\to W$, where $\tilde W$ is the blowup of~$a$;
we identify $Z$ with $\tilde W$ and $F$ with~$Q$ via this
isomorphism. In the quadric $F\subset Z$, let $\ell\subset F$ be a
fiber of the projection $\sigma|_F\colon F\to C$, and let $m\subset F$
be a ``line of the ruling'' that is mapped onto $C$ by~$\sigma$. We
supposed that $W_1$ is a projective variety; if $H$ is an ample
divisor on $W_1$, then $(s^*H,\ell)=0$ and $(s^*H, m)>0$, so the
images of the fundamental classes of $\ell$ and~$m$ in $H_2(Z,\Q)\cong
H_2(\tilde W,\Q)$ are not proportional. Hence, the rank of the mapping
$H_2(Q,\Q)\to H_2(\tilde W,\Q)$ induced by the embedding
$Q\hookrightarrow \tilde W$ is equal to~$2$, which contradicts our
hypothesis.\qed
\end{proof}

To prove assertion~(ii) of Proposition~\ref{res-nores} we will need
yet another lemma.

\begin{lemma}\label{NE_1+NE_2}
Suppose that $X$ is a smooth projective threefold and $F\subset X$ is
an irreducible surface.

Let $\NE_1\subset N_1(X)$ be the cone of effective $1$-cycles on $X$
such that all their components lie in $F$, and let
$\NE_2\subset N_1(X)$ be the cone of effective $1$-cycles on $X$ such
that none of their components lies in $F$. In this setting, if
$\NEb_i$, $i=1,2$, is the closure of $\NE_i$, then
$\NEb(X)=\NEb_1+\NEb_2$.
\end{lemma}

\begin{proof}
Observe first that $\NEb_1+\NEb_2$ is closed in $N_1(X)$.  Indeed, if
$A,B\subset\R^n$ are closed convex cones such that $A\cap(-B)=\{0\}$,
then their sum $A+B\subset\R^n$ is closed (see for example
\cite[Theorem 2.1]{Klee} or \cite[Theorem 3.2]{Beutner}), so to prove
this assertion it suffices to show that
$\NEb_1\cap(-\NEb_2)=\{0\}$. Now if $u\in \NEb_1\cap(-\NEb_2)$ and if
$H$ is an ample divisor on~$X$, then $(u, H)\ge0$ since $u\in \NEb_1$,
and $(u, H)\le 0$ since $u\in -\NEb_2$, whence $(u, H)=0$. Since $u\in
\NEb(X)$, Kleiman's criterion for amplitude (see \cite[Theorem
  1.4.29]{Lazarsfeld}) implies that $u=0$, and we are done.

Since, by the very construction,
$\NE(X)=\NE_1+\NE_2\subset \NEb_1+\NEb_2$ and since $\NEb_1+\NEb_2$ is
closed, one has $\NEb(X)\subset \NEb_1+\NEb_2$. Taking into account
that $\NEb_1,\NEb_2\subset \NEb(X)$, we are done.
\end{proof}

\begin{proof}[of assertion~\textup{(ii)} of
    Proposition~\ref{res-nores}] 
It is clear that one can embed $\tilde W$ in a projective space so
that the ``lines'' of two rulings on the quadric
$Q=\sigma^{-1}(a)\subset\tilde W$ are embedded as actual lines; to that
end, it suffices to consider a projective embedding $W\subset\PP^N$, put
\begin{equation*}
\widetilde{\PP^N}=\{(x,L)\in \PP^N\times\PP^{N-1}\colon x\in L\},
\end{equation*}
where $\PP^{N-1}$ is the set of lines in $\PP^N$ passing
through~$a$, and observe that
\begin{equation*}
\tilde W=\overline{p^{-1}(W\setminus\{a\})}\subset \widetilde{\PP^N},  
\end{equation*}
where $p\colon \widetilde{\PP^N}\to\PP^N$ is induced by the projection
on the first factor. Till the end of the proof we will fix a
projective embedding of $\tilde W$ with these properties.

Let $\ell,m\subset Q$ be lines of two different rulings. We claim that
$\R_+[\ell]$ is an extremal ray in $\NEb(\tilde W)$.

Observe that, since $\Oo_Q(Q)\cong \Oo_Q(-1,-1)$, one has
$\omega_{\tilde W}\otimes\Oo_Q\cong \Oo_Q(-1,-1)$, so $(\ell,
K_{\tilde W})=-1<0$. Thus, to prove that $\R_+[\ell]$ is an extremal
ray it remains to show that if $[\ell]=u+v$, where $u,v\in\NEb(\tilde
W)$, then $u$ and $v$ are multiples of~$[\ell]$.

Suppose now that $[\ell]=u+v$, 
$u,v\in\NEb(\tilde W)$, and
put $X=\tilde W$ and $F=Q$ in Lemma~\ref{NE_1+NE_2}. One concludes
that $u=u_1+u_2$ and $v= v_1+v_2$, where, using the notation of
the above-mentioned lemma, $u_1,v_1\in \NEb_1$ and $u_2,v_2\in\NEb_2$. Putting
$w_1=u_1+v_1\in\NEb_1$, $w_2=u_2+v_2\in\NEb_2$, one has $[\ell]=w_1+w_2$.

Let $H$ be the divisor class of a hyperplane section of~$\tilde W$. 
Intersecting $\ell$ with $H$ and
with~$Q$, one has
\begin{equation}
  \begin{aligned}\label{w.Q}
  1&=(w_1, H)+(w_2, H),\\
  -1&=(w_1, Q)+(w_2, Q).  
  \end{aligned}
\end{equation}
Since $w_1$ is a linear combination of curves lying in $Q$ and
$\Oo_Q(Q)\cong \Oo_Q(-H)$, one has $(w_1, Q)=-(w_1, H)$. Adding
the equations~\eqref{w.Q} one obtains $(w_2, H)+(w_2,
Q)=0$. However, $(w_2, H)\ge0$ since $H$ is ample and $(w_2,
Q)\ge0$ since no component of~$w_2$ lies in~$Q$. Hence, $(w_2,
H)=(u_2, H)+(v_2, H)=0$. Since $H$ is ample, $(u_2,
H)\ge0$ and $(v_2, H)\ge0$, whence $u_2=v_2=0$ by Kleiman's
criterion~\cite[Theorem 1.4.29]{Lazarsfeld}.

Thus, $u=u_1$, $v=v_1$, so $u,v\in\NEb_1$.  Since $\Pic(Q)$ is
generated by the classes of $\ell$ and~$m$, one has
$\NEb_1=\{a[\ell]+b[m]\colon a,b\ge0\}$. Observe now that $[\ell]$ and
$[m]$ are linearly independent in $N_1(X)$. Indeed, image of the
fundamental class of a projective curve $C\subset\tilde W$ in the
second rational homology group of $\tilde W$ is uniquely determined by
intersection indices of $C$ with divisors. Now if $[\ell]$ and $[m]$
are linearly dependent, then images of fundamental classes of $\ell$
and $m$ are proportional in $H_2(\tilde W,\R)$, so $\rank i_*=1$,
which contradicts our hypothesis. Thus, if $u=a[\ell]+b[m]$,
$v=a'[\ell]+b'[m]$, and $[\ell]=u+v$, one has $b+b'=0$. Since
$b,b'\ge0$, this implies that $b=b'=0$, so $u$ and $v$ are
multiples of $v$, as required.

Thus, $\R_+[\ell]$ is an extremal ray. Since $[\ell]$ is not a
multiple of~$[m]$, this ray is of the type~(1.1.2.1) in the notation
of~\cite{Kollar}. Hence, there exists a birational mapping
$s\colon\tilde W\to W_1$, where $W_1$ is smooth and $s^{-1}$ defines
the blowing up of a curve $C\subset W_1$ such that $C\cong \PP^1$ and
$N_{W_1|C}\cong \Oo_{\PP^1}(-1)\oplus\Oo_{\PP^1}(-1)$. The rational
mapping $s_1=\sigma\circ s^{-1}\colon W_1\dasharrow W$ is continuous
in the classical topology; since $W_1$ is smooth, the mapping $s_1$ is
a morphism. Hence, $s_1\colon W_1\to W$ is the desired small
resolution.\qed
\end{proof}

\section{Some consequences of Condition (A)}\label{sec:top2geom}

Suppose that $X\subset\PP^N$ is a smooth projective variety of
dimension~$4$ for which $X^*\subset(\PP^N)^*$ is a hypersurface. Let
$Y_0$ be a hyperplane section with a single singular point which is a
node.

If $\tilde Y$ is the blowup of $Y_0$ at the singularity, let $Q\subset
\tilde Y$ be the inverse image of the singularity; $Q$ is isomorphic
to the smooth 2-dimensional quadric. Let $i_*\colon H_2(Q,\Q)\to
H_2(\tilde Y,\Q)$ be the homomorphism induced by the embedding
$Q\hookrightarrow \tilde Y$.

\begin{proposition}\label{b_4(Y_0)}
In the above setting, if Condition~\textup(A\textup) holds for $X$,
then $\rank i_*=1$, and if Condition~\textup(A\textup) does not hold
for $X$, then $\rank i_*=2$.
\end{proposition}

\begin{proof}
A general smooth hyperplane section $Y\subset X$ can be included
in a Lefschetz pencil one of the fibers of which is~$Y$.  By $S\subset
Y$ we denote the vanishing cycle (homeomorphic to the sphere~$S^3$) that is
contracted to an $A_1$ singularity, so that $Y_0$ is homeomorphic
to~$Y/S$.  It is well known (see \cite{Lamotke} or \cite[Expos\'es
  XVII et XVIII]{SGA7.2}) that the monodromy group acting in $H^3(Y,\Q)$
is generated by (pseudo)reflections in the classes of vanishing cycles
corresponding to singular fibers of the Lefschetz pencil, and that all
these classes are conjugate by the action of the monodromy
group. Thus, Condition~(A) for the variety~$X$ is violated if and only
if the vanishing cycle~$S$ is homologous to zero in $H_3(Y,\Q)$.

Consider the following two fragments of the exact sequence of the pair
$(Y,S)$: 
\begin{gather}
H_2(S,\Q)\to H_2(Y,\Q)\to H_2(Y_0,\Q)\to H_1(S,\Q),\label{eq-of-hom2}\\
H_4(S,\Q)\to H_4(Y,\Q)\to H_4(Y_0,\Q)\to H_3(S,\Q)\xrightarrow{j_*}
H_3(Y,\Q)\label{eq-of-hom}
\end{gather}
where $j\colon S\to Y$ is the natural embedding.

It follows from \eqref{eq-of-hom2} that $b_2(Y)=b_2(Y_0)$, and it
follows from~\eqref{eq-of-hom} and Poincar\'e duality for~$Y$ that
$b_4(Y_0)=b_2(Y_0)$ if Condition (A) holds and that
$b_4(Y_0)=b_2(Y_0)+1$ if Condition (A) does not hold. Now the
desired result follows from Proposition~\ref{b_2=b_4}.\qed
\end{proof}

\section{Conclusion}\label{conclusion}

Now we can prove Propositions~\ref{omega(1)} and~\ref{not-A}. 

\begin{proof}[of Proposition~\ref{not-A}]
Suppose that $X\subset\PP^n$ is a smooth $4$-dimensional projective
variety over~\C such that
$X^*\subset(\PP^n)^*$ is a hypersurface. It follows from \cite[Theorem
  17]{Kleiman} that a general singular hyperplane section of~$X$ has
precisely one singular point, and this point is a node.

Suppose now that
 Condition~\textup{(A)} holds for~$X$; let
$Y_0\subset X$ be a hyperplane section with a unique singularity $a\in
Y_0$ which is a node. It follows from Proposition~\ref{b_4(Y_0)}
that $\rank i_*=1$, where $i\colon Q\to \tilde Y_0$ is the natural
embedding, $\sigma \colon \tilde Y_0\to Y_0$ is the blowup of~$a$,
and $Q=\sigma^{-1}(a)$. Then Proposition~\ref{res-nores} implies the
non-existence of a small resolution for $Y_0$.  

On the other hand, if Condition~(A) does not hold for~$X$, then, keeping the
previous notation, it follows from Proposition~\ref{b_4(Y_0)} that
$\rank i_*=2$, whence, by virtue of Proposition~\ref{res-nores}, a
projective small resolution for~$Y_0$ exists.\qed
\end{proof}

\begin{proof}[of Proposition~\ref{omega(1)}]
Suppose that $X\subset\PP^n$ is a smooth projective fourfold such that
$H^0(X,\omega_X(1))\ne 0$. If $Y\subset X$ is a smooth hyperplane
section, then it follows from the implication
$\mathrm{iii})\Rightarrow \mathrm{iv})$ in \cite[Proposition
  6.1]{Lvovski} that $b_3(Y)>b_3(X)$, in particular, the restriction
homomorphism $H^3(X,\Q)\to H^3(Y,\Q)$ cannot be surjective. A theorem
of A.\,Landman (see \cite[Theorem 22]{Kleiman}) asserts that if
$X\subset \PP^n$ is a smooth $d$-dimensional projective manifold such
that $X^*\subset(\PP^n)^*$ is not a hypersurface, then
$b_{d-1}(Y)=b_{d-1}(X)$. Thus, if $H^0(X,\omega_X(1))\ne 0$ then $X^*$
is a hypersurface and Condition~(A) is satisfied for~$X$. Now
Proposition~\ref{not-A} applies.\qed
\end{proof}

\bibliographystyle{amsplain}      

\bibliography{notabib}

\providecommand{\bysame}{\leavevmode\hbox to3em{\hrulefill}\thinspace}
\providecommand{\MR}{\relax\ifhmode\unskip\space\fi MR }
\providecommand{\MRhref}[2]{%
  \href{http://www.ams.org/mathscinet-getitem?mr=#1}{#2}
}
\providecommand{\href}[2]{#2}
\begin{thebibliography}{10}

\bibitem{Atiyah}
M.~F. Atiyah, \emph{On analytic surfaces with double points}, Proc. Roy. Soc.
  London Ser. A \textbf{247} (1958), 237--244. \MR{95974}

\bibitem{Beutner}
Eric Beutner, \emph{On the closedness of the sum of closed convex cones in
  reflexive {B}anach spaces}, J. Convex Anal. \textbf{14} (2007), no.~1,
  99--102. \MR{2310430}

\bibitem{Cheltsov}
Ivan Cheltsov, \emph{On factoriality of nodal threefolds}, J. Algebraic Geom.
  \textbf{14} (2005), no.~4, 663--690. \MR{2147353}

\bibitem{SGA7.2}
P.~Deligne and N.~Katz, \emph{Groupes de monodromie en g\'eom\'etrie
  alg\'ebrique. {II}}, Lecture Notes in Mathematics, Vol. 340, Springer-Verlag,
  Berlin, 1973, S{\'e}minaire de G{\'e}om{\'e}trie Alg{\'e}brique du Bois-Marie
  1967--1969 (SGA 7 II), Dirig{\'e} par P. Deligne et N. Katz. \MR{0354657 (50
  \#7135)}

\bibitem{GelMan}
Sergei~I. Gelfand and Yuri~I. Manin, \emph{Methods of homological algebra},
  second ed., Springer Monographs in Mathematics, Springer-Verlag, Berlin,
  2003. \MR{1950475}

\bibitem{H}
Robin Hartshorne, \emph{Algebraic geometry}, Springer-Verlag, New
  York-Heidelberg, 1977, Graduate Texts in Mathematics, No. 52. \MR{0463157}

\bibitem{Klee}
V.~L. Klee, Jr., \emph{Separation properties of convex cones}, Proc. Amer.
  Math. Soc. \textbf{6} (1955), 313--318. \MR{68113}

\bibitem{Kleiman}
Steven~L. Kleiman, \emph{Tangency and duality}, Proceedings of the 1984
  {V}ancouver conference in algebraic geometry (Providence, RI), CMS Conf.
  Proc., vol.~6, Amer. Math. Soc., 1986, pp.~163--225. \MR{846021 (87i:14046)}

\bibitem{Kollar}
J\'{a}nos Koll\'{a}r, \emph{Extremal rays on smooth threefolds}, Ann. Sci.
  \'{E}cole Norm. Sup. (4) \textbf{24} (1991), no.~3, 339--361. \MR{1100994}

\bibitem{Lamotke}
Klaus Lamotke, \emph{The topology of complex projective varieties after {S}.
  {L}efschetz}, Topology \textbf{20} (1981), no.~1, 15--51. \MR{592569}

\bibitem{Lazarsfeld}
Robert Lazarsfeld, \emph{Positivity in algebraic geometry. {I}. {C}lassical
  setting: line bundles and linear series}, Ergebnisse der Mathematik und ihrer
  Grenzgebiete. 3. Folge. A Series of Modern Surveys in Mathematics [Results in
  Mathematics and Related Areas. 3rd Series. A Series of Modern Surveys in
  Mathematics], vol.~48, Springer-Verlag, Berlin, 2004. \MR{2095471}

\bibitem{Lvovski}
Serge Lvovski, \emph{On curves and surfaces with projectively equivalent
  hyperplane sections}, Canad. Math. Bull. \textbf{37} (1994), no.~3, 384--392.
  \MR{1289775 (95i:14012)}

\bibitem{PRS}
Francesco Polizzi, Antonio Rapagnetta, and Pietro Sabatino, \emph{On
  factoriality of threefolds with isolated singularities}, Michigan Math. J.
  \textbf{63} (2014), no.~4, 781--801. \MR{3286671}

\bibitem{Reid}
Miles Reid, \emph{Minimal models of canonical {$3$}-folds}, Algebraic varieties
  and analytic varieties ({T}okyo, 1981), Adv. Stud. Pure Math., vol.~1,
  North-Holland, Amsterdam, 1983, pp.~131--180. \MR{715649}

\bibitem{Zak}
F.~L. Zak, \emph{Surfaces with zero {L}efschetz cycles}, Mat. Zametki
  \textbf{13} (1973), 869--880 (Russian), English translation: Math. Notes 13
  (1973), 520--525. \MR{0335517 (49 \#298)}

\end{thebibliography}

\end{document}